\documentclass[11pt]{article}

\usepackage{geometry}
\usepackage[T1]{fontenc}
\usepackage[utf8]{inputenc}
\usepackage{authblk}
\usepackage{amsmath}
\usepackage{amssymb}
\usepackage{amsfonts}
\usepackage{amsthm}
\usepackage{enumerate}
\usepackage{paralist}
\usepackage{amsfonts}
\usepackage{tensor} 
\usepackage{lmodern}
\usepackage{upgreek}
\usepackage{mathrsfs}
\usepackage{nicefrac}
\usepackage{verbatim}
\usepackage{amsmath}
\usepackage[all, cmtip]{xy}\usepackage{amssymb}
\usepackage{amsthm}
\usepackage{extpfeil}
\usepackage{bbm}
\usepackage{paralist}
\usepackage[colorlinks,citecolor=blue,pagebackref=true,urlcolor=blue,pdftex]{hyperref}

\usepackage{titlesec}
\titleformat*{\section}{\large\bfseries\sffamily}

\DeclareMathOperator{\dps}{depth}

\theoremstyle{plain}
\newtheorem{thm}{Theorem}[section]

\newtheorem{prop}[thm]{Proposition}
\newtheorem{lem}[thm]{Lemma}
\newtheorem{cor}[thm]{Corollary}

\theoremstyle{definition}

\theoremstyle{remark}

\renewcommand{\k}{\mathbbm{k}}

\begin{document}
% ------------------------------------------------------------------------
\title{Clique Vectors of $k$-Connected Chordal Graphs}

% ------------------------------------------------------------------------

\author{Afshin Goodarzi%
  \thanks{Email address: \texttt{afshingo@kth.se}}}
\affil{{\small Royal Institute of Technology, Department of Mathematics, S-100 44, Stockholm, Sweden.}}

\maketitle

%\affil{, }
%\email{}

%\author{Afshin Goodarzi%
  %\thanks{Email address: \texttt{afshingo@kth.se}}}
%\affil{Royal Institute of Technology,\\ Department of Mathematics,\\ S-100 44, Stockholm, %Sweden.}

\begin{center}

\textit{Dedicated to Ralf Fr\"{o}berg on the occasion of his 70th birthday}

\vspace{12pt}
\end{center}

\begin{abstract}
The clique vector $\mathfrak{c}(G)$ of a graph $G$ is the sequence $(c_1, c_2, \ldots,c_d)$ in $\mathbb{N}^d$, where $c_i$ is the number of cliques in $G$ with $i$ vertices and $d$ is the largest cardinality of a clique in $G$. In this note, we use tools from commutative algebra to characterize all possible clique vectors of $k$-connected chordal graphs. 
\end{abstract}

\section{Introduction}

The clique vector of a graph $G$ is an interesting numerical invariant assigned to $G$. The study of clique vectors goes back at least to Zykov's generalization of Tur\'{a}n's graph theorem~\cite{Z}. The clique vector of $G$ is by definition the $f$-vector of its clique complex. Challenging problems including the Kalai--Eckhoff conjecture and the classification of the $f$-vector of flag complexes led many researchers to investigate clique vectors, see for instance \cite{F1}, \cite{F2}, and \cite{HHMTZ}. While the Kalai--Eckhoff conjecture is now settled by Frohmader~\cite{F1}, the latter problem is still wide open. 

Herzog et~al. \cite{HHMTZ} characterized all possible clique vectors of chordal graphs. A graph $G$ is called \emph{$k$-connected} if it has at least $k$ vertices and removing any set of vertices of $G$ of cardinality less than $k$ yields a connected graph. Thus a $1$-connected graph is simply a connected graph. We use the convention that every graph is $0$-connected. The \emph{connectivity number} $\kappa(G)$ of $G$ is the maximum number $k$ such that $G$ is $k$-connected. The aim of this paper is to characterize all possible clique vectors of $k$-connected chordal graphs. More precisely we prove the following result.

\begin{thm}\label{main}
A vector $\mathfrak{c}=\left(c_1,\ldots,c_d\right)\in\mathbb{N}^d$ is the clique vector of a $k$-connected chordal graph if and only if the vector $\mathfrak{b}=\left(b_1,\ldots,b_d\right)$ defined by 
\begin{equation}\label{eq1}
\sum_{1}^d b_i x^{i-1}=\sum_{1}^d c_i (x-1)^{i-1}
\end{equation} 
has positive components and $b_1=b_2=\ldots=b_k=1$. 
\end{thm}
The theorem above is a refinement of~\cite[Theorem 1.1]{HHMTZ}, in the sense that putting $k=0$, the only requirement on $b$-numbers is to be positive, so \cite[Theorem 1.1]{HHMTZ} will be obtained.

In order to prove our main result, we shall use techniques from algebraic shifting theory to reduce the problem to the class of shifted graphs, the so called threshold graphs.

The rest of this paper is organized as follows. In Section~\ref{s2}, we verify the characterization for threshold graphs by giving a combinatorial interpretation of the $b$-numbers. Section~\ref{s3} is devoted to a study of the connectivity of a graph via certain homological invariants of a ring associated to it. Finally, in Section~\ref{s4} we prove our main result.

All undefined algebraic terminology can be found in the book of Herzog and Hibi~\cite{HH}.

\section{Clique Vectors of Threshold Graphs}\label{s2}

Let $G$ be a graph. We denote by $S(G)$ the graph obtained from $G$ by adding a new vertex and connecting it to all vertices of $G$. Also, we denote by $D(G)$ the graph obtained from $G$ by adding an isolated vertex. Clearly the numbers of $i$-cliques in $G$ and $D(G)$ are the same, unless $i=1$. On the other hand, it is easy to verify the following formula that relates the numbers of cliques in $G$ and $S(G)$:
\begin{equation}\label{cone}
1+\sum_i c_i\left(S(G)\right)x^i=\left(1+\sum_i c_i(G)x^i \right)(1+x).
\end{equation}
A graph $T$ is called \emph{threshold}, if it can be obtained from the null graph by a sequence of $S$- and $D$-operators. Thus, we have a bijection between threshold graphs and words on the alphabet $\{S,D\}$ (reading from left to right) with an $S$ in its final (rightmost) position\footnote{The $D$- and $S$-operations on the null graph, i.e. the graph having zero vertices, result the same graph. So, to have a unique representation of each threshold graph, we may assume that the operation $D$ is allowed when the graph is not null.}. Clearly, every threshold graph is chordal.

Many properties of a threshold graph can be read off from its word. Among them are the following simple but useful facts. 
\begin{lem}\label{l1}
Let $T$ be a threshold graph. Then the following hold:
\begin{compactenum}[\rm (i)]
\item The number of times that $S$ appears in $T$ is the clique number of $T$.
\item $T$ is $k$-connected if and only if there is no $D$ in the first $k$ letters of $T$.
\end{compactenum}
\end{lem}\qed

Let $T$ be a threshold graph with clique number $d$. We put a $/$ right after every $S$ in the word of $T$, thus breaking the word of $T$ into $d$ subwords. Let $b_i$ be the length of the $i$-th subword. Then the $b$-vector of $T$ is $\mathfrak{b}(T)=(b_1,b_2,\ldots, b_d)$. For instance, if $T=DDDSSDSDDS$, then $T$ breaks to $DDDS/S/DS/DDS/$ and $\mathfrak{b}(T)=(4,1,2,3)$. 

It turns out that knowing the $b$-vector of a threshold graph is equivalent to knowing its clique vector.

\begin{prop}
Let $T$ be a threshold graph. Then the clique vector $\mathfrak{c}(T)=(c_1,\ldots, c_d)$ can be obtained from $\mathfrak{b}(T)=(b_1,\ldots, b_d)$ using the formula 
\begin{equation}\label{eq:3}
\sum_{i=1}^d b_i(x+1)^{i-1}=\sum_{i=1}^d c_i x^{i-1}.
\end{equation}
\end{prop}

\begin{proof}
The statement is clear if $T$ is an isolated vertex, so we may inductively assume that it has been proved for threshold graphs on $n-1$ vertices. Suppose that $T$ is a threshold graph on $n$ vertices. Then $T$ is either $D(T')$ or $S(T')$, for a threshold graph $T'$. In the former case the statement follows easily from the induction hypothesis. \\
Suppose $T=S(T')$. Then $b_1=1$ and $\mathfrak{b}(T')=(b_2,\ldots,b_d)$. The induction hypothesis and equation~\eqref{cone} imply that
\[
1+\sum_i c_i(T)x^i=\left(1+x\sum_{i=2}^d b_i(x+1)^{i-2} \right)(1+x).
\]
Therefore equation~\eqref{eq:3} follows. 
\end{proof}

Let $\mathcal{B}(n,d,k)$ denote the set of all positive-integer vectors $(b_1,b_2,\ldots,b_d)$ such that $\sum b_i =n$ and $b_1=b_2=\cdots=b_k=1$. 
The set of $k$-connected threshold graphs on $n$ vertices and clique number $d$ is denoted by $\mathcal{T}(n,d,k)$. The mapping $T\mapsto \mathfrak{b}(T)$ is an injection from $\mathcal{T}(n,d,k)$ into $\mathcal{B}(n,d,k)$, by Lemma~\ref{l1}. A small computation, left to the reader, shows that the sets $\mathcal{T}(n,d,k)$ and $\mathcal{B}(n,d,k)$ have the same cardinality ${n-k-1}\choose{d-k-1}$. So, $T\mapsto \mathfrak{b}(T)$ is indeed a bijective correspondence between $\mathcal{T}(n,d,k)$ and $\mathcal{B}(n,d,k)$. Now putting this all together, we can conclude the following characterization of clique vectors of $k$-connected threshold graphs. 

\begin{cor}\label{spc}
A vector $\mathfrak{c}=\left(c_1,\ldots,c_d\right)\in\mathbb{N}^d$ is the clique vector of a $k$-connected threshold graph if and only if the vector $\mathfrak{b}=\left(b_1,\ldots,b_d\right)$ defined by equation~\eqref{eq1} has positive components and $b_1=b_2=\ldots=b_k=1$. 
\end{cor}\qed

\section{Algebraic Tools}\label{s3}

Let $\Gamma$ be a simplicial complex on the vertex set $[n]$. Let $\k$ be a field of characteristic zero and $R=\k[x_1,\ldots,x_n]$ the polynomial ring on $n$ variables. The \emph{Stanley-Reisner ideal} $I_\Gamma$ of $\Gamma$ is the ideal in $R$ generated by all monomials $x_{i_1}\cdots x_{i_l}$, where $\{i_1,\ldots,i_l\}$ is not a face of $\Gamma$. The \emph{face ring} $\k[\Gamma]$ of $\Gamma$ is the quotient ring $R/I_\Gamma$.  

Let $G$ be a graph on the vertex set $[n]$. The collection $\Delta(G)$ of the cliques in $G$ forms a simplicial complex, known as the \emph{clique complex} of $G$. Clique complexes are \emph{flag}, that is, all minimal non-faces have the same cardinality two. Moreover, every flag complex is the clique complex of its underlying graph ($1$-skeleton).

In this section we study the connectivity number of a chordal graph via a homological invariant, namely the bigraded Betti numbers (see e.g.~\cite[Appendix A]{HH}) of the face ring of its clique complex. We start with a general result.

\begin{thm}\label{betti}
A graph $G$ is $k$-connected if and only if $\beta_{i,i+1}(\k[\Delta(G)])=0$ for all $i\geq n-k$. In particular, $\kappa(G)=\max\{ k\mid \beta_{i,i+1}(\k[\Delta(G)])=0 \text{ for all } i\geq n-k \}$.
\end{thm} 
\begin{proof}
By Hochster's formula~\cite[Theorem 8.1.1]{HH}
\begin{equation*}
\beta_{i,i+1}(\k[\Delta(G)])=\sum_{\mid W\mid=i+1}\dim_\k \widetilde{H}_0\left(\Delta(G)_W\right).
\end{equation*}
On the other hand, the induced subcomplex $\Delta(G)_W$ is the clique complex of the induced graph $G_W$. So, $\beta_{i,i+1}(\k[\Delta(G)])=0$ if and only if $G_W$ is connected for all $W$ of cardinality $i+1$. Now, since the induced subgraph on a set $W$ is the same as the graph obtained by removing the complement $\overline{W}$ of $W$ from $G$, it follows that $\beta_{i,i+1}(\k[\Delta(G)])=0$ if and only if removing any set of $n-i-1$ vertices results in a connected graph. Therefore $\beta_{i,i+1}(\k[\Delta(G)])=0$ for all $i\geq n-k$ if and only if removing any set of at most $k-1$ vertices leaves a connected graph, as desired.

\end{proof}
The theorem above gives a general lower bound for the connectivity number of the graph.
\begin{cor}\label{dps}
If $G$ is a graph, then $\dps (\k[\Delta(G)])\leq \kappa(G)+1$.
\end{cor}
\begin{proof}
If the projective dimension of $\k[\Delta(G)]$ is $p$, then $\beta_{i,i+1}(\k[\Delta(G)])=0$ for all $i\geq p+1$. Thus, Theorem~\ref{betti} gives the lower bound of $n-p-1$ for $\kappa(G)$. And therefore the result follows from Auslander--Buchsbaum formula~\cite[Corollary A.4.3]{HH}. 
\end{proof}
In the rest of this section, we show that the bound obtained in Corollary~\ref{dps} is sharp as it is realized for the chordal graphs. The following fundamental result of Ralf Fr\"{o}berg plays an essential role in the rest of this paper.
\begin{thm}[Fr\"{o}berg~\cite{Ralf}]\label{Ralf} Let $\Gamma$ be a simplicial complex. Then $\Gamma$ is the clique complex of a chordal graph if and only if $\k[\Gamma]$ has a $2$-linear resolution, i.e. $\beta_{i,j}(\k[\Gamma])=0$, whenever $(i,j)\neq (0,0)$ and $j-i\neq 1$.

\end{thm}
\begin{cor}\label{chdps}
If $G$ is a chordal graph, then $\dps (\k[\Delta(G)])= \kappa(G)+1$.
\end{cor}
\begin{proof}
If $G$ is a chordal graph, then by Fr\"{o}berg's Theorem~\ref{Ralf}, we have $\beta_{i,j}(\k[\Delta(G)])=0$, whenever $j-i\neq 1$. So, the projective dimension is equal to the maximum $p$ such that $\beta_{p,p+1}(\k[\Delta(G)])\neq 0$. It now follows from Theorem~\ref{betti} that $p+1=n-\kappa(G)$. 
\end{proof}

\section{Main Result}\label{s4}
In this section, we prove our main result by using techniques from shifting theory.

A simplicial complex $\Gamma$ on the vertex set $[n]$ is \emph{shifted} if, for $F\in\Gamma$, $i\in F$, $j\notin F$ and $j<i$ the set $(F\setminus\{i\})\cup\{j\}$ is a face of $\Gamma$. A shifted complex is flag if and only if it is clique complex of a threshold graph~\cite[Theorem 2]{K}. \emph{Exterior algebraic shifting} is an operation $\Gamma\rightarrow\Gamma^e$, associating to a simplicial complex $\Gamma$ a shifted simplicial complex $\Gamma^e$, while preserving many interesting algebraic, combinatorial and topological invariants and properties. We refer the reader to the book by Herzog and Hibi~\cite{HH} for the precise definition and more information. Here we mention some of the properties that will be used later.

\begin{lem}\label{shift}
Let $\Gamma$ be a simplicial complex. Then the following hold.
\begin{compactenum}[\rm (i)]
\item\label{part1} Exterior shifting preserves the $f$-vector; $f(\Gamma)=f(\Gamma^e)$.
\item Alexander duality and exterior shifting commute; $(\Gamma^{*})^e=(\Gamma^e)^*$.
\item Exterior shifting preserves the depth; $\dps(\k[\Gamma])=\dps(\k[\Gamma^e])$. 
\end{compactenum}

\end{lem}\qed

The following result is known and has been used in the literature, see e.g. \cite[Theorem 3.1]{GY}. However, for the convenience of the reader, we supply a proof.
\begin{lem}\label{lm}
Let $\Gamma$ be a flag complex. Then $\Gamma$ is the clique complex of a chordal graph if and only if its exterior shifting $\Gamma^e$ is the clique complex of a threshold graph.
\end{lem}

\begin{proof} We show the ``only if'' direction. The other direction follows by reversing the proof sequence. Suppose that $\Gamma=\Delta(G)$ for some chordal graph $G$. Next, Fr\"{o}berg's Theorem~\ref{Ralf} implies that $\k[\Gamma]$ has a $2$-linear resolution. Thus, it follows from Eagon--Reiner Theorem~\cite[Theorem 8.19]{HH}, that the Alexander dual $\Gamma^{*}$ of $\Gamma$ is Cohen--Macaulay of dimension $n-3$. So, $(\Gamma^e)^{*}$ is Cohen--Macaulay of the same dimension, since exterior algebraic shifting commutes with Alexander duality and preserves Cohan--Macaulayness. Hence, the theorems of Eagon--Reiner and Fr\"{o}berg imply that $\Gamma^e$ is the clique complex of a chordal graph $T$. Now, since $\Gamma^e$ is flag and shifted,  $T$ is a threshold graph. 
\end{proof}
Now we are in the position to prove our main result.
\begin{proof}[Proof of Theorem~\ref{main}]
Let $G$ be a $k$-connected chordal graph. Let us denote by $G^e$ the threshold graph such that $\Delta(G^e)=\Delta(G)^e$. It follows from part~\eqref{part1} of Lemma~\ref{shift} that $\mathfrak{c}(G)=\mathfrak{c}(G^e)$. On the other hand, since $\dps(\k[\Delta(G)])=\dps(\k[\Delta(G^e)])$, Corollary~\ref{chdps} implies that $G^e$ is $k$-connected. Therefore the result follows from Corollary~\ref{spc}.
\end{proof}
\paragraph*{Acknowledgments.} I am grateful to Anders Bj\"{o}rner, Anton Dochtermann and Siamak Yassemi for valuable comments. I would also like to thank the anonymous referees whose
comments led to an improvement of the paper.

{\small
\def\cprime{$'$}
\providecommand{\bysame}{\leavevmode\hbox to3em{\hrulefill}\thinspace}
\providecommand{\MR}{\relax\ifhmode\unskip\space\fi MR }
\providecommand{\MRhref}[2]{%
  \href{http://www.ams.org/mathscinet-getitem?mr=#1}{#2}
}
\providecommand{\href}[2]{#2}

}
\end{document}